\documentclass[11pt,reqno]{amsart}
\usepackage{amssymb}
\usepackage{amsmath}
\usepackage{mathrsfs}
\usepackage{amssymb, amsmath, amsfonts}
\usepackage{bbm}
\usepackage{mathabx,url}
\usepackage{graphicx}
\usepackage{color}

\usepackage[hidelinks]{hyperref}

\makeatletter
\def\tank#1{\protected@xdef\@thanks{\@thanks
 \protect\footnotetext[0]{#1}}}
\def\bigfoot{

 \@footnotetext}
\makeatother

\newcommand{\ea}{\end{array}}

\allowdisplaybreaks
\numberwithin{equation}{section}
\newtheorem{theorem}{Theorem}[section]
\newtheorem{lemma}{Lemma}[section]
\newtheorem{proposition}{Proposition}[section]

\newtheorem{cor}{Corollary}[section]

\newtheorem{definition}{Definition}[section]

\def\beq{\begin{equation}}
\def\nneq{\end{equation}}

\def\bthm{\begin{theorem}}
\def\nthm{\end{theorem}}

\def\blem{\begin{lemma}}
\def\nlem{\end{lemma}}
\def\bprf{\begin{proof}}
\def\nprf{\end{proof}}
\def\bprop{\begin{prop}}
\def\nprop{\end{prop}}
\def\brmk{\begin{rem}}
\def\nrmk{\end{rem}}

\def\bexa{\begin{exa}}
\def\nexa{\end{exa}}
\def\bcor{\begin{cor}}
\def\ncor{\end{cor}}

\title[Lower classes and Chung's LIL of the  FIGFBM]{Lower classes and Chung's LILs of the fractional integrated generalized fractional Brownian motion}

\author[M. Lyu]{Mengjie Lyu}
\address[]{Mengjie Lyu, School of Mathematics and Statistics,  Wuhan University,  Wuhan, 430072,
China.}
\email{mengjie.lyu@whu.edu.cn}

\author[M. Wang]{Min Wang}
\address[]{Min Wang, School of Mathematics and Statistics,  Wuhan University,  Wuhan, 430072,
China.}
\email{minwang@whu.edu.cn}

\author[R. Wang]{Ran Wang}
\address[]{Ran Wang, School of Mathematics and Statistics,  Wuhan University,  Wuhan, 430072,
China.}
\email{rwang@whu.edu.cn}

\date{}
\begin{document}
\maketitle

 \noindent{\bf Abstract:}{
Let $\{X(t)\}_{t\geqslant0}$ be the generalized fractional Brownian motion  introduced by Pang and Taqqu  (2019):
 \begin{align*}
 \{X(t)\}_{t\ge0}\overset{d}{=}&\left\{  \int_{\mathbb R}  \left((t-u)_+^{\alpha}-(-u)_+^{\alpha} \right) |u|^{-\gamma/2}
  B(du)  \right\}_{t\ge0},
 \end{align*}
where $ \gamma\in [0,1),  \ \  \alpha\in \left(-\frac12+\frac{\gamma}{2}, \  \frac12+\frac{\gamma}{2} \right)$ are constants.
For any $\theta>0$, let
\begin{align*}
 Y(t)=\frac{1}{\Gamma(\theta)}\int_0^t (t-u)^{\theta-1} X(u)du, \quad t\ge 0.
\end{align*}
Building upon the arguments of Talagrand (1996),    we give integral criteria for   the lower classes of $Y$  at $t=0$ and at infinity, respectively. As a consequence, we derive its Chung-type laws of the iterated logarithm.  In the proofs,   the small ball probability estimates  play important roles. }
 \vskip0.3cm

 \vskip0.3cm

 \vskip0.3cm

 \section{Introduction and main results}


The generalized fractional Brownian motion (GFBM, for  short) $X:=\{X(t)\}_{t\ge0}$ is a centered Gaussian  self-similar
process, which was introduced by Pang and Taqqu \cite{PT2019} as the scaling limit of power-law shot noise processes. It has
a stochastic integral representation:
 \begin{align}\label{eq X}
 \{X(t)\}_{t\ge0}\overset{d}{=}&\left\{  \int_{\mathbb R}  \left((t-u)_+^{\alpha}-(-u)_+^{\alpha} \right) |u|^{-\gamma/2}
  B(du)  \right\}_{t\ge0},
 \end{align}
 where the parameters $\gamma$ and $\alpha$ satisfy
 \begin{align}\label{eq constant}
 \gamma\in [0,1),  \ \  \alpha\in \left(-\frac12+\frac{\gamma}{2}, \  \frac12+\frac{\gamma}{2} \right),
 \end{align}
and $B(du)$ is a Gaussian random measure in $\mathbb R$ with the Lebesgue  control measure $du$.  It follows that the Gaussian process $X$ is self-similar with index $H$ given by
 \begin{align}
 H=\alpha-\frac{\gamma}{2}+\frac12\in(0,1).
 \end{align}

If $\gamma=0$, then $X$ becomes an ordinary fractional Brownian motion  (FBM, for short) $B^H$, which can be represented as:
\begin{align}\label{eq FBM}
\left\{B^H(t)\right\}_{t\ge 0}\overset{d}{=}\left\{ \int_{\mathbb R}
\left((t-u)_+^{H-\frac12}-(-u)_+^{H-\frac12} \right) B(du)  \right\}_{t\ge0}.
\end{align}
 If $\gamma\in (0,1)$, as shown by Pang and Taqqu \cite{PT2019},  GFBM   preserves the self-similarity property while
the factor $|u|^{-\gamma}$ introduces non-stationarity of  the increments, which is useful for reflecting the
non-stationarity of random noises  in certain physical systems.   For example, Ichiba, Pang and Taqqu \cite{IPT2020b}
studied the semimartingale properties of GFBM, and applied them to model the volatility  processes in finance.

 For the GFBM, Ichiba, Pang and Taqqu   \cite{IPT22} studied its   H\"older continuity, the functional and local laws of the iterated logarithm  and other  sample path properties;  Wang and Xiao \cite{WX2022a}
studied its precise local sample path properties, including the exact uniform modulus of continuity, small ball probabilities, Chung's law of the iterated logarithm (Chung's LIL, for short) at any fixed point $t>0$.    Chung's LIL at the origin   for GFBM   is quite subtle.  Wang and Xiao \cite{WX2022b} studied this problem  by modifying the arguments of Talagrand  \cite{Tal96}.

The results in \cite{IPT22,WX2022a, WX2022b}    show that GFBM   is an interesting example of self-similar Gaussian processes. It   has richer sample path properties than the ordinary FBM and its close relatives such as the Riemann-Liouville FBM (cf.    \cite{El11}), bifractional Brownian motion (cf.  \cite{TX07}), and the sub-fractional Brownian motion (cf.  \cite{El12}).
 In this sense, GFBM is a good object  that can be studied for the purpose of developing a general theoretical framework for studying the fine properties of self-similar Gaussian processes which, to the best of our knowledge, is still not complete yet. In particular, the asymptotic properties of a general self-similar Gaussian process at $t=0$ could be quite subtle and difficult to characterize precisely (see, e.g., \cite[Remark  1.1]{WX2022b}).

In this paper,    we continue this line of research and consider   the fractional  integrated generalized fractional Brownian motion (FIGFBM, for short):
 \begin{align}\label{y}
 Y(t)=\frac{1}{\Gamma(\theta)}\int_0^t (t-u)^{\theta-1} X(u)du, \quad t\geqslant0,\, \theta>0,
\end{align}
where  $\Gamma(\theta)=\int_{0}^{\infty } x^{\theta -1}e^{-x}dx$ is the Gamma function.      If $X$ is a FBM, $Y(t)$ is called a fractional integrated fractional Brownian motion; See,  e.g.,  \cite{El03, LL99}. If further $\theta\in\mathbb N$ and $X$ is a Brownian motion,  then  $Y=\{Y(t)\}_{t\ge0}$ is called the ($\theta$-1)-fold  integrated Brownian motion; See, e.g.,   \cite{CL03, ZL2006}.


In this paper,  we   characterize the  lower classes of $Y$ by some integral tests. As a consequence, we derive the  Chung's LILs of $Y$ at $t = 0$ and at infinity, respectively.  This  approach  was initiated in the pioneering work by  Chung \cite{Chung} for Brownian motion,    further developed by Talagrand \cite{Tal96} for FBM,    and later  by El-Nouty  \cite{El03}-\cite{El12} for some close relatives of FBM.     Building upon this  argument,  Wang and Xiao  \cite{WX2022b}  derived the Chung's LILs  of GFBM $X$ at $t = 0$ and at infinity, respectively.

In order to give precise statement of our results, we first recall the  definitions of the lower classes for the processes,  which  were introduced in
\cite{Rev90}. We refer to \cite{Rev90} for a systematic and extensive account on the studies of lower
and upper classes for Brownian motion and random walks.

 \begin{definition}\label{Def-LowerClasses}
 Let $\big\{M(t)\big\}_{t\ge0}$
 be a stochastic process and $\xi(t)$ be a function on $(0, \infty)$.
 \begin{itemize}
 \item[(a)] The function $\xi(t)$ belongs to the lower-lower class of the process $M$ at  $\infty$
 (resp. at $0$), denoted by $\xi\in LLC_{\infty}(M)$ (resp. $\xi\in LLC_{0}(M)$), if for almost all
 $\omega\in \Omega$ there exists $t_0=t_0(\omega)$ such that $M(t)\ge \xi(t)$ for every $t>t_0$
 (resp. $t<t_0$).

\item[(b)]The function $\xi(t)$ belongs to the lower-upper class of the process $M$ at  $\infty$
(resp. at $0$), denoted by $\xi\in LUC_{\infty}(M)$ (resp. $\xi\in LUC_{0}(M)$), if for  almost all $\omega\in
\Omega$ there exists a sequence $0<t_1(\omega)<t_2(\omega)<\cdots$ with $t_n(\omega)\uparrow\infty$
(resp.  $t_1(\omega) > t_2(\omega)>\cdots$ with $t_n(\omega)\downarrow0$), as $n\rightarrow \infty$,
such that $M(t_n(\omega)) \le \xi(t_n(\omega)), n\in \mathbb N$.
\end{itemize}
\end{definition}

In the following,  we always assume that
\begin{equation}\label{Def M}
    M(t):= \sup _{0\leq s\leq t}|Y(s)|,
\end{equation}
and consider its  small ball function
\begin{equation}\label{def varphi}
\varphi (\varepsilon) := \mathbb P\left(M(1) \leq\varepsilon\right) = \mathbb P\left( M(t) \leq \varepsilon t^{H+\theta}\right).
\end{equation}
 The following two theorems characterize the lower classes of FIGFBM $Y$ at $t = 0$ and $\infty$, respectively.

\begin{theorem}\label{Thm-LowerClassesZero}
Let $\gamma \in[0,1)$, $\alpha \in(-1 / 2+\gamma/2, 1 / 2)$, $\theta>0$ and $\xi:\left(0, e^e\right] \rightarrow(0, \infty)$  be a nondecreasing continuous function.
\begin{itemize}
\item[(a)] (Sufficiency). If
\begin{equation}\label{3.15}
 \frac{\xi(t)}{t^{H+\theta}}
 \text{ is bounded and }
 I_{0}(\xi):=\int_{0}^{e^e}\left(\frac{\xi(t)}{t^{H+\theta}}\right)^{-1 / \beta} \varphi \left(\frac{\xi(t)}{t^{H+\theta}}\right) \frac{d t}{t}<+\infty,
 \end{equation}
then $\xi \in L L C_{0}(M)$.

\item[(b)] (Necessity). If  $\xi \in L L C_{0}(M)$  and there exists a constant  $c_{1,1}\ge  1$  such that  $\xi(2 t) \le  c_{1,1} \xi(t)$  for all  $t \in\left(0, e^e/ 2\right]$, then \eqref{3.15} holds.
\end{itemize}
Here and below, we always assume that $\beta=\alpha+\theta+1/2$ and $H=\alpha-\gamma/2+1/2$.
\end{theorem}

\begin{theorem}\label{Thm-LowerClassesInf}
Let  $\gamma \in[0,1)$, $\alpha \in(-1 / 2+\gamma/2, 1 / 2)$, $\theta>0$ and $\xi:\left[e^e, \infty\right) \rightarrow(0, \infty)$  be a nondecreasing continuous function. Then $ \xi \in L L C_{\infty}(M)$  if and only if
\begin{equation}\label{3.9}
 \frac{\xi(t)}{t^{H+\theta}}
 \text{ is bounded and }
 I_{\infty}(\xi):=\int_{e^e}^{\infty}\left(\frac{\xi(t)}{t^{H+\theta}}\right)^{-1 / \beta} \varphi \left(\frac{\xi(t)}{t^{H+\theta}}\right) \frac{d t}{t}<+\infty.
 \end{equation}
 \end{theorem}

The proofs of Theorems \ref{Thm-LowerClassesZero} and \ref{Thm-LowerClassesInf} are based on a modification of the Talagrand's   arguments in \cite{El03, Tal96,  WX2022b}, which  will be given in Section \ref{Sec Proofs} below.

Combining Theorems \ref{Thm-LowerClassesZero} and   \ref{Thm-LowerClassesInf}  with Lemma \ref{le3}  below, we  obtain the following  results.

\begin{theorem}\label{Thm-ChungLIL}
Assume $\gamma\in[0,1)$, $\alpha\in(-1/2 +\gamma/2
, 1/2)$ and $\theta>0$.
\begin{itemize}
\item[(a)] There exists a positive constant  $\kappa_{1}\in(0, \infty)$  such that
\begin{equation}\label{3.1}
\liminf _{t \rightarrow 0+} \sup _{0 \le  s \le  t} \frac{|Y(s)|}{t^{H+\theta} /\left(\ln \ln t^{-1}\right)^{\beta }}=\kappa_{1} \quad \text { a.s. }
\end{equation}

\item[(b)]  There exists a positive constant  $\kappa_{2}\in(0, \infty)$  such that
\begin{equation}\label{3.2}
\liminf _{t \rightarrow \infty} \sup _{0 \le  s \le  t} \frac{|Y(s)|}{t^{H+\theta} /(\ln \ln t)^{\beta}}=\kappa_{2} \quad \text { a.s. }
\end{equation}
\end{itemize}
\end{theorem}

Limit theorems of the forms \eqref{3.1} and \eqref{3.2}  are known as   Chung's LIL, which was initial studied by  Chung \cite{Chung}.   We refer to Bingham \cite{Bin86},   Monrad and Rootz\'en \cite{MR95},  Li and Shao \cite{LiS01}, Buchmann and Maller  \cite{BM11} and the references therein for more information.

It is known that small ball probability estimates are essential for studying the lower classes of a stochastic process and Chung's LILs (cf. \cite{Rev90, Tal96}).  In the proofs of our main theorems,  the  following small ball probability estimates  play important roles.

\begin{theorem}\label{Thm-SmallBall} Let $\gamma\in[0,1)$, $\alpha\in (-1/2+\gamma/2, 1/2)$ and $\theta>0$. Then there exist some constants $\kappa_3, \kappa_4\in (0,\infty)$  such that for any $\varepsilon\in (0,1)$,
  \begin{align}\label{eq log varphi}
  \kappa_{3}  \varepsilon^{-\frac{1}{\beta}} \le  - \log\mathbb P\left(\sup_{t\in [0,1]}|Y(t)| \le \varepsilon \right)\le  \kappa_{4}   \varepsilon^{-\frac{1}{\beta}}.
 \end{align}
\end{theorem}
  The small ball probabilities are not only useful for proving lower classes in Theorems \ref{Thm-LowerClassesZero} and   \ref{Thm-LowerClassesInf},   and Chung's LIL in Theorem \ref{Thm-ChungLIL},  but also have many other applications.
  We refer to Lifshits \cite{Lif99} and Li and Shao \cite{LiS01} for  more information.  Particularly,
 Kuelbs and Li \cite{KuL93} discovered a remarkable link between the small ball estimate and the metric entropy number of the unit balls of reproducing kernel Hilbert space generated by a general Gaussian measure.  Li and Linde \cite{LL99} successfully removed  the   regularity assumption imposed on the unknown function in the link, which was assumed in  \cite{KuL93}.  By using the approach in \cite{LL99}   and the   small ball estimates of GFBM   in \cite{WX2022a},  we obtain   the lower bound of the small ball probability in Theorem \ref{Thm-SmallBall}.

The rest of this paper is organized as follows.  In Section  \ref{Proof of Thm-SmallBall},  we prove the  small ball probability estimates of $Y$ in Theorem \ref{Thm-SmallBall}.
  In Section \ref{Sec Proofs}, we prove two important  maximal inequalities, based on which Theorems \ref{Thm-LowerClassesZero} and \ref{Thm-LowerClassesInf}  can be proved by following the same lines in \cite{Tal96, WX2022b}.    Section \ref{Proof of Thm-ChungLIL} is devoted to the proofs of  Chung's LILs in Theorem \ref{Thm-ChungLIL} by using  some zero-one laws for the lower class of $Y$.

\section{Small ball probability estimates}
\label{Proof of Thm-SmallBall}
In this section, we prove  the small ball probability estimates of $Y$  presented in Theorem \ref{Thm-SmallBall}. As mentioned in the introduction, we derive the upper bound of the small ball probability of $Y$   by using the strong local non-determinism property (see, e.g.,   Monard and Rootz\'en \cite{MR95}), and obtain the lower bound by means of the link between the small ball probability and the metric entropy established  by Kuelbs and Li \cite{KuL93} and Li and Linde \cite{LL99}.  Let us prove the upper and lower bounds of the small ball probability of $Y$, respectively.
 \subsection{Upper bound}
Let us first prove the   the upper estimates of small ball probability of $Y$ by using the strong local  non-determinism property (see, e.g.,    Monard and Rootz\'en \cite{MR95},  Xiao \cite{Xiao2008})
\begin{lemma}\label{lem upper} Let $\gamma\in[0,1)$, $\alpha\in (-1/2+\gamma/2, 1/2)$ and $\theta>0$. Then there exists some constant  $\kappa_3 \in (0,\infty)$  such that for any $\varepsilon\in (0,1)$,
  \begin{align}\label{eq log varphi upper}
   \mathbb P\left(\sup_{t\in [0,1]}|Y(t)| \le \varepsilon \right)\le  \exp\left( - \kappa_{3}  \varepsilon^{-\frac{1}{\beta}} \right).
    \end{align}
\end{lemma}
\begin{proof}
Notice that
\begin{equation*}
\begin{split}
Y(t) =&\, \frac{1}{\Gamma(\theta)} \int_{0}^{t} (t-u)^{\theta-1}du \int_{-\infty}^{0}B(d w) \left((u-w)^{\alpha}-(-w)^{\alpha}\right)(-w)^{-\frac{\gamma}{2}} \\
&+\frac{1}{\Gamma(\theta)} \int_{0}^{t}(t-u)^{\theta-1}du \int_{0}^{u}B(dw)(u-w)^{\alpha} w^{-\frac{\gamma}{2}}  \\
=:&\, Y_{1}(t)+Y_{2}(t).
\end{split}
\end{equation*}
By using the Anderson inequality \cite{And1955}, we have
\begin{equation}\label{2.5}
\begin{split}
\mathbb{P}\left(\sup _{t \in[0,1]}|Y(t)| \leq \varepsilon\right) & =\mathbb{E}\left[\mathbb{P}\left(\sup _{t \in[0,1]}|Y_{1}(t)+Y_{2}(t)| \leq \varepsilon \mid Y_{1}(t), t \in[0,1]\right)\right] \\
& \leq \mathbb{P}\left(\sup _{t \in[0,1]}|Y_{2}(t)| \leq \varepsilon\right).
\end{split}
\end{equation}
For any $0\le  h\le  1$ and $0\le  t \le  1-h$, by stochastic Fubini theorem, we have
\begin{equation*}
\begin{split}
Y_{2}(t+h) =&\,\frac{1}{\Gamma(\theta)} \int_{0}^{t+h}(t+h-u)^{\theta-1}du  \int_{0}^{u}B(d w)(u-w)^{\alpha} w^{-\frac{\gamma}{2}} \\
 =&\,\frac{1}{\Gamma(\theta)} \int_{0}^{t+h} w^{-\frac{\gamma}{2}}B(d w)  \int_{w}^{t+h}du(t+h-u)^{\theta-1}(u-w)^{\alpha}\\
=&\, \frac{\Gamma(\alpha+1)}{\Gamma(\alpha+\theta+1)}\int_{0}^{t+h}(t+h-w)^{\alpha+\theta} w^{-\frac{\gamma}{2}}B(d w),
\end{split}
\end{equation*}
where  a change of variable with $u =t+h-(t+h-w)x$ is used in the last step.

 We further  divide   $Y_2$ into two terms:
\begin{equation}\label{Z2}
\begin{split}
Y_{2}(t+h)
 = \,  \frac{\Gamma(\alpha+1)}{\Gamma(\alpha+\theta+1)}\left[\int_{0}^{t}+\int_{t}^{t+h}\right](t+h-w)^{\alpha+\theta} w^{-\frac{\gamma}{2}}B(d w).
\end{split}
\end{equation}
Here,  the first integral in \eqref{Z2} is measurable with respect to  $\sigma(B(s): 0\le  s \leq t)$ and the second integral is independent of  $\sigma(B(s): 0\le  s \leq t)$. Since  $\sigma(Y_2(s): 0\le  s \leq t) \subseteq \sigma(B(s): 0\le  s \leq t)$, we have  that for any $0\leqslant t+h \leqslant 1$,
\begin{align}\label{eq SLND Y2}
\begin{split}
\operatorname{Var}\left(Y_{2}(t+h) \mid Y_{2}(s): 0 \leq s \leq t\right) & \ge  \operatorname{Var}\left(Y_{2}(t+h) \mid B(s): 0 \leq s \leq t\right) \\
& = \frac{\Gamma^{2}(\alpha+1)}{\Gamma^{2}(\alpha+\theta+1)}\int_{t}^{t+h}(t+h-w)^{2(\alpha+\theta)} w^{-\gamma}dw  \\
& \ge   \frac{\Gamma^{2}(\alpha+1)}{2\beta\Gamma^{2}(\alpha+\theta+1)}(t+h)^{-\gamma} h^{2\beta}\\
& \ge  \frac{\Gamma^{2}(\alpha+1)}{2\beta\Gamma^{2}(\alpha+\theta+1)} h^{2\beta}.
\end{split}
\end{align}
 Since conditional distributions in Gaussian processes are Gaussian, it follows from \eqref{eq SLND Y2} that
\begin{align*}
     \mathbb{P}\big( \left|Y_{2}(k/n)\right| \leq   \varepsilon \; |\; Y_{2}(j/n)=x_j, \,1\le j\le k-1 \big)
 \le   \,     \Phi\left(c\varepsilon n^{\beta}\right)- \Phi\left(-c\varepsilon n^{\beta} \right),
\end{align*}
where $\Phi$ is the distribution function  of a standard normal random variable, and $c= \frac{\sqrt{2\beta}\Gamma (\alpha+\theta+1)}{\Gamma (\alpha+1)}$. Thus, by repeated conditioning,
 \begin{align*}
\mathbb{P}\left(\sup _{t \in[0,1]}|Y_{2}(t)| \leq \varepsilon\right) \le &\,  \mathbb{P}\left(\sup _{1\le k\le n}\left|Y_{2}\left({k}/{n}\right)\right| \leq \varepsilon\right)\\
 \le  &\,   \left( \Phi\left(c\varepsilon n^{\beta}\right)- \Phi\left(-c\varepsilon n^{\beta} \right)   \right)^n.
\end{align*}
Choosing $n=\lceil(c\varepsilon)^{-1/\beta}\rceil\ge 2$, we get that
\begin{align*}
    \mathbb{P}\left(\sup _{t \in[0,1]}|Y_2(t)| \leq \varepsilon\right)\le \exp\left( - \kappa_{3}  \varepsilon^{-\frac{1}{\beta}} \right).
\end{align*}
This, together with \eqref{2.5}, implies the inequality \eqref{eq log varphi upper}. The proof is complete.
  \end{proof}

\subsection{Lower  bound} Let $\mu$ denote a centered Gaussian measure on a real separable Banach space $E$ with norm $||\cdot||$ and dual $E^*$. Recall that if $(E, d)$ is a  metric space and $A$ is a compact subset of $(E, d)$, then the metric entropy of $A$ is $\log N(A, \varepsilon)$,  where $N(A, \varepsilon)$ is the minimum covering number,
$$N(A, \varepsilon) := \min \left\{n \geq 1: \exists \, x_{1}, \ldots, x_{n} \in A \text { s.t.} \bigcup_{j=1}^{n} B_{\varepsilon}(x_j)  \supseteq A\right\},$$
and $ B_{\varepsilon}(x_0) = \left\{x: d(x,x_0)<\varepsilon\right\} $ is the open ball of radius $\varepsilon$ centered at $x_0$. The Hilbert space $H_\mu$ generated by $\mu$ can be described as the completion of the range of the mapping $S: E^* \rightarrow E$ defined via the Bochner integral,
$$ Sa = \int_E x a(x) d\mu(x), \quad a \in E^*,$$
and the completion is in the inner product norm
$$ \langle Sa, Sb\rangle_\mu =  \int_E a(x) b(x)d\mu(x), \quad a, b \in E^* .  $$
Lemma 2.1 in \cite{Kue76} presents the details of this construction; See also Ledoux and Talagrand \cite{LT91}.
Since the unit ball $B_\mu$ of $H_\mu$ is always compact, $B_\mu$ has a finite metric entropy.
For any compact operator $v$ from a Banach space $E$ into another one, we borrow the definition of the $n$-th dyadic entropy number of $v$ from \cite[Page 1560]{LL99}, \begin{equation*}   e_n(v) = \inf \left\{ \varepsilon > 0: \,  N(v(B_E), \varepsilon) \leq 2^{n-1} \right\}. \end{equation*}
Let $X = (X(t))_{t\in T}$ be a centered Gaussian process with a.s. continuous sample paths and let $K: [0,1]^2 \rightarrow \mathbb{R}$ be a measurable kernel such that
$$(I_K f)(t) = \int_0^1 K(t,s)f(s)ds $$
is an operator from $C[0,1]$ into itself.
We now present a very useful result for the lower bound of small ball probability of $Y$, which is a continuous centered Gaussian stochastic process constructed by
$$ Y(t) = \int_0^1 K(t,s)X(s)ds. $$
We write $f(x) \preceq g(x)$ as $x \rightarrow 0$ if $\lim\sup_{x\rightarrow 0} f(x)/g(x) < \infty.$
\begin{proposition}{\cite[Theorem 6.1]{LL99}}
\label{prop thm6.1 in LL99}
    If there exist some constants  $a > 0$, $r > 0$ and $b, p \in \mathbb{R}$ such that
\begin{equation*}
    -\log  \mathbb P\left(\sup_{t\in [0,1]} |X(t)|\le  \varepsilon \right) \preceq \varepsilon^{-a} \left( \log \frac{1}{\varepsilon}\right)^b
\end{equation*}
    and
    $$ e_n(I_K) \preceq n^{-1/r}(1+\log n)^p, $$
    then
    \begin{equation*}
    -\log  \mathbb P\left(\sup_{t\in [0,1]} |Y(t)|\le  \varepsilon \right) \preceq \varepsilon^{-ar/(a+r)} \left( \log \frac{1}{\varepsilon}\right)^{(ap+b)r/(a+r)}.
\end{equation*}
\end{proposition}
Next, we use Proposition \ref{prop thm6.1 in LL99} to prove the lower bound of small ball probability of FIGFBM $Y$, which is reformulated as the following Lemma.

\begin{lemma}\label{lem lower} Let $\gamma\in[0,1)$,  $\alpha\in (-1/2+\gamma/2, 1/2)$ and $\theta>0$.  Then there exists a constant  $\kappa_4 \in (0,\infty)$  such that for any $\varepsilon\in (0,1)$,
  \begin{align}\label{eq log varphi lower}
   \mathbb P\left(\sup_{t\in [0,1]}|Y(t)| \le \varepsilon \right)\ge \exp\left( - \kappa_{4}  \varepsilon^{-\frac{1}{\beta}} \right).
    \end{align}
\end{lemma}

\begin{proof}
 Theorem 1.2 in \cite{WX2022a} says that there exists a  positive constant
   $\kappa_{5}$ such that for all $ 0<\varepsilon<1$,
 \begin{align}\label{eq small X}
  \mathbb P\left(\sup_{t\in [0,1]} |X(t)|\le  \varepsilon \right)\ge
  \exp\left( - \kappa_{5}  \varepsilon^{-\frac{1}{\alpha + 1/2}} \right).
 \end{align}
For any $\theta>0$,  the Riemann-Liouville integral operator $R_{\theta} : C([0,1],\|\cdot\|_{\infty}) \rightarrow  C([0,1],\|\cdot\|_{\infty})$ is defined by
$$
\left(R_{\theta} f\right)(t)=\frac{1}{\Gamma(\theta)} \int_{0}^{t}(t-s)^{\theta-1} f(s) d s, \quad 0 \leq t \leq 1. $$
Here, $C([0,1],\|\cdot\|_{\infty})$ denotes the space of continuous functions from $[0,1]$ to $\mathbb R$ with the sup-norm $\|\cdot\|_{\infty}$.  Obviously,  $Y(t)=\left(R_{\theta} X\right)(t)$
with the GFBM $X$ being defined as in \eqref{eq X}.

Set
$$
e_{n}\left(R_{\theta}: C([0,1]) \rightarrow  C([0,1]) \right)=\inf \left\{\varepsilon>0: N\left(R_{\theta}(C([0,1])), \varepsilon\right)\le 2^{n-1}\right\}.
$$
Here,  $N(R_{\theta}(C([0,1])), \varepsilon)$ is called  the minimum covering number which equals that:
$$\min \left\{n \geq 1: \exists x_{1}, \ldots, x_{n} \in R_{\theta}(C([0,1])) \text { s.t.} \bigcup_{j=1}^{n} \left\{x: \|x- x_{j}\|_{\infty}<\varepsilon\right\}  \supseteq R_{\theta}(C([0,1]))\right\}.$$
 By Proposition 6.1 in \cite{LL99}, there exists   a positive constant $c_{2,1}$ such that
\begin{align}\label{en}
    c_{2,1}^{-1} n^{-\theta}\le
    e_{n}\left(R_{\theta}:  C([0,1],\|\cdot\|_{\infty} )  \rightarrow  C([0,1]),\|\cdot\|_{\infty} \right) \le c_{2,1} n^{-\theta}.
\end{align}
Combining \eqref{eq small X} with \eqref{en}, and using Proposition \ref{prop thm6.1 in LL99}, we  obtain the inequality \eqref{eq log varphi lower}.

The proof is complete.
  \end{proof}

 \subsection{Small ball   functions}
Set
\begin{align}\label{eq varphi}
\psi(\varepsilon):=-{\log} \varphi (\varepsilon).
\end{align}
Then $\psi$ is positive and non-increasing. According to Borell \cite{2},  $\psi$ is convex which implies the
existence of the right derivative $\psi^{\prime}$ of $\psi$. Thus, $\psi^{\prime}\le 0$ and $|\psi^{\prime}|$ is non-increasing.

From the small ball probability estimates in \eqref{eq log varphi},
we can see that there exists a constant  $K_{1} \ge  1$ such that for all  $\varepsilon<1$,
\begin{equation}\label{eq K1}
\frac{1}{K_{1} \varepsilon^{1/\beta}} \le  \psi(\varepsilon) \le  \frac{K_{1}}{\varepsilon^{1/\beta}}.
\end{equation}

The following lemmas give more properties of the functions $\varphi$ and $\psi$,   which are similar to those in Talagrand \cite[Section 2]{Tal96}.
 Since  the following three lemmas   follow along the same lines in the  proof of the analogous lemmas for FBM in \cite[Section 2]{Tal96},  we omit the details here.

\begin{lemma}\label{pro 4.1}
There exists a  positive constant $K_{2}$ such that for all $\varepsilon\in (0, 1/K_{2})$,
\begin{equation*}
-\frac{K_{2}}{\varepsilon^{1+1 /\beta}}\le  \psi^{\prime}(\varepsilon) \le  -\frac{1}{K_{2} \varepsilon^{1+1 /\beta}}  .
\end{equation*}
\end{lemma}

\begin{lemma}\label{cor4.1}
There exists a constant $K_{3}$  such that for all $\varepsilon \in (0, 2/K_2) $ and $\mu > \varepsilon/2 $,
\begin{equation}\label{6}
\exp \left(-K_{3} \frac{|\mu-\varepsilon|}{\varepsilon^{1+1/\beta}}\right) \le  \frac{\varphi\left(\mu\right)}{\varphi(\varepsilon)} \le  \exp \left(K_{3} \frac{|\mu-\varepsilon|}{\varepsilon^{1+1/\beta}}\right).
\end{equation}
\end{lemma}

\begin{lemma}\label{lem increase}
For all $\varepsilon< (\beta/K_2)^{\beta}$, the function $\varepsilon^{-1/\beta}\varphi(\varepsilon)$ is increasing.
\end{lemma}

\section{Proofs of Theorems  \ref{Thm-LowerClassesZero} and \ref{Thm-LowerClassesInf}}
\label{Sec Proofs}
By using the  properties of the small ball functions  in Lemmas \ref{pro 4.1}-\ref{lem increase} and by using the following two maximal inequalities in Propositions \ref{lem-MtMu} and \ref{Prop-MtMu1}  below, the proofs of Theorems  \ref{Thm-LowerClassesZero} and \ref{Thm-LowerClassesInf} follow the same lines  in   \cite[Sections 3 and 4]{WX2022b} except  with $H$  symbol   replaced by  $H+\theta$,   To avoid the repetition with \cite{WX2022b}, we only give the full proofs for Propositions \ref{lem-MtMu} and \ref{Prop-MtMu1}.
\subsection{One maximal inequality}
\begin{proposition}\label{lem-MtMu}
If $\gamma\in[0,1)$,  $\alpha\in(-1/2+\gamma/2,1/2)$ and $\theta>0$, then for all $0<t<u$ and  $\nu$, $\eta>0$, there exists a positive constant $c_{3,1}$ such that
\begin{equation}\label{eq 41}
\mathbb{P}\left(\left\{M(t) \leq \nu t^{H+\theta}\right\} \cap\{M(u) \leq \eta\}\right)
\leq  2 \varphi (\nu)\exp \left(-\frac{u-t}{c_{3,1} u^\frac{\gamma}{2\beta}\eta^{\frac{1}{\beta}}}\right).
\end{equation}
\end{proposition}
\begin{proof}
When $(u-t) /\left(u^{\frac{\gamma}{2\beta}} \eta^{\frac{1}{\beta}}\right) \le  2$, it is obvious that \eqref{eq 41} holds with $c_{3,1}=2/\ln2$. Thus,  it is sufficient to  prove \eqref{eq 41} in  the case of  $(u-t) /\left(u^{\frac{\gamma}{2\beta}} \eta^{\frac{1}{\beta}}\right)>2$.  The proof is divided into two steps.

\noindent {\bf Step 1.} We define an increasing sequence  $\left\{t_{k}\right\}_{k \ge  0}$ as follows. Set  $t_{0}=t$. For any $k\ge0$, if $t_{k}$ has been defined, then we choose $t_{k+1}$ such that
\begin{equation}\label{eq tk}
  t_{k+1}-t_{k+1}^{\frac{\gamma}{2\beta}} \eta^{\frac{1}{\beta}}=t_{k}.
 \end{equation}
 Consider the event
\begin{equation}\label{G}
G_{k}:= \left\{M(t) \leq \nu t^{H+\theta}\right\} \bigcap\left\{ M(t_k)  \leq \eta\right\}.
\end{equation}
 To prove \eqref{eq 41}, it suffices to prove that for any $k\ge1$,  it holds that
 \begin{align}\label{eq Gk}
 \mathbb P(G_k)\le \varphi(\nu)\rho^k,
 \end{align}
 where $\rho \in  (0, 1)$ is a constant that depends on $\alpha, \gamma$ and $\theta$ only. Indeed, if
  $k_{0}$  is the largest integer such that  $t_{k_{0}} \le  u$, then $t_{k_{0}+1}>u$.  Consequently, by \eqref{eq tk}, we  have
\begin{equation*}
t_{k_{0}}=t_{k_{0}+1}-t_{k_{0}+1}^{\frac{\gamma}{2\beta}} \eta^{\frac{1}{\beta}}>u-u^{\frac{\gamma}{2\beta}} \eta^{\frac{1}{\beta}}.
\end{equation*}
This, together with fact that $t_{k+1}-t_{k}=t_{k+1}^{\frac{\gamma}{2\beta}} \eta^{\frac{1}{\beta}} \le  u^{\frac{\gamma}{2\beta}} \eta^{\frac{1}{\beta}}$ for any $k<  k_{0}$, implies that
\begin{equation*}
k_{0}+1 \ge  \frac{t_{k_{0}}-t}{u^{\frac{\gamma}{2\beta}} \eta^{\frac{1}{\beta}}}+1=\frac{t_{k_{0}}+u^{\frac{\gamma}{2\beta}} \eta^{\frac{1}{\beta}}-t}{u^{\frac{\gamma}{2\beta}} \eta^{\frac{1}{\beta}}}>\frac{u-t}{u^{\frac{\gamma}{2\beta}} \eta^{\frac{1}{\beta}}}.
\end{equation*}
Hence, we know that   $k_{0}>(u-t) /\left(2 u^{\frac{\gamma}{2\beta}} \eta^{\frac{1}{\beta}}\right)$  and
\begin{equation*}
 G_{k_{0}} \supseteq\left\{M(t) \leqslant \nu t^{H+\theta}\right\} \cap\{M(u) \leqslant \eta\}.
\end{equation*}
Consequently, \eqref{eq Gk} implies \eqref{eq 41}.

\noindent {\bf Step 2.}
We prove \eqref{eq Gk} by induction over  $k$. The result holds for  $k=0$  by \eqref{def varphi}. For the induction step, we observe that
$$G_{k+1} \subseteq G_{k} \cap\{|I| \leqslant 2 \eta\},$$
where
\begin{equation*}
\begin{split}
I :=&\, Y\left(t_{k+1}\right)-Y\left(t_{k}\right) \\
=&\,\frac{1}{\Gamma(\theta)} \int_{0}^{t_{k+1}}\left(t_{k+1}-w\right)^{\theta-1} X(w) dw-\frac{1}{\Gamma(\theta)} \int_{0}^{t_{k}}\left(t_{k}-w\right)^{\theta-1} X(w) dw \\
=&\, \frac{1}{\Gamma(\theta)} \int_{t_{k}}^{t_{k+1}}\left(t_{k+1}-w\right)^{\theta-1} X(w) d w\\
&+\frac{1}{\Gamma(\theta)} \int_{0}^{t_{k}}\left(\left(t_{k+1}-w\right)^{\theta-1}-\left(t_{k}-w\right)^{\theta-1}\right) X(w) d w. \\
\end{split}
\end{equation*}
Recall $X(w)$ defined in \eqref{eq X}. The integral $I$ can be rewritten as $I=I_{1}+I_{2}$, where
\begin{equation*}
\begin{split}
I_{1}:=&\, \frac{1}{\Gamma(\theta)} \int_{t_{k}}^{t_{k+1}}\left(t_{k+1}-w\right)^{\theta-1}d w  \int_{t_{k}}^{w}B(d x)(w-x)^{\alpha} x^{-\frac{\gamma}{2}}\\
I_{2}:= &\frac{1}{\Gamma(\theta)} \int_{t_{k}}^{t_{k+1}}\left(t_{k+1}-w\right)^{\theta-1}d w \int_{-\infty}^{t_k}B(d x)\left((w-x)^{\alpha}-(-x)_+^\alpha\right) |x|^{-\frac{\gamma}{2}} \\
&+\frac{1}{\Gamma(\theta)} \int_{0}^{t_{k}}\left(\left(t_{k+1}-w\right)^{\theta-1}-\left(t_{k}-w\right)^{\theta-1}\right) X(w) d w. \\
\end{split}
\end{equation*}
By stochastic Fubini theorem, we have
\begin{equation*}
\begin{split}
I_{1} & =\frac{1}{\Gamma(\theta)} \int_{t_{k}}^{t_{k+1}} x^{-\frac{\gamma}{2}}B(dx)  \int_{x}^{t_{k+1}}dw\left(t_{k+1}-w\right)^{\theta-1}(w-x)^{\alpha}\\
&=\frac{1}{\Gamma(\theta)} \int_{t_{k}}^{t_{k+1}} x^{-\frac{\gamma}{2}}  \left(t_{k+1}-x\right)^{\alpha+\theta}B(dx) \int_{0}^{1} dy  (1-y)^{\theta-1} y^{\alpha}\\
&=\frac{\Gamma(\alpha+1)}{\Gamma(\alpha+\theta+1)} \int_{t_{k}}^{t_{k+1}}\left(t_{k+1}-x\right)^{\alpha+\theta} x^{-\frac{\gamma}{2}} B(dx),
\end{split}
\end{equation*}
where the change of variable $y=(w-x)/(t_{k+1}-x)$ has been used in the second step.
The integral representation of $I_{1}$ implies that
$\mathbb{E}\left(I_{1}\right)=0$  and
\begin{equation*}
\begin{split}
\operatorname{Var}(I_{1})&
=\frac{\Gamma^{2}(\alpha+1)}{\Gamma^{2}(\alpha+\theta+1)} \int_{t_{k}}^{t_{k+1}}\left(t_{k+1}-x\right)^{2(\alpha+\theta)} x^{-\gamma} dx\\
&\ge  \frac{\Gamma^{2}(\alpha+1)}{2\beta\Gamma^{2}(\alpha+\theta+1)}t_{k+1}^{-\gamma}\left(t_{k+1}-t_{k}\right)^{2\beta} \\
&=\frac{\Gamma^{2}(\alpha+1)}{2\beta\Gamma^{2}(\alpha+\theta+1)}\eta^{2}.
\end{split}
\end{equation*}
This implies that
\begin{equation*}
\begin{split}
\mathbb P\left(\left|I_{1}\right| \le  2 \eta\right)\le&\,   \Phi\left(\frac{2 \sqrt{2 \beta} \Gamma(\alpha+\theta+1)}{\Gamma(\alpha+1)}\right)-\Phi\left(-\frac{2 \sqrt{2 \beta} \Gamma(\alpha+\theta+1)}{\Gamma(\alpha+1)}\right)
=:  \rho,
\end{split}
\end{equation*}
where  $\Phi $  denotes the distribution function of  a standard Gaussian random variable.
Consequently, by   the  independence of $I_{1}$  and $ \sigma\left\{B(s) ; s \le  t_{k}\right\}$,  the fact of $I_2\in   \sigma\left\{B(s) ; s \le  t_{k}\right\}$ and by Anderson's inequality  \cite{And1955},
we have
\begin{equation*}
\begin{split}
\mathbb{P}\left(G_{k+1}\right)   \le &\,  \mathbb{E}\Big[\mathbb{P}\left(G_{k} \cap\left\{\left|I_{1}+I_{2}\right| \le  2 \eta\right\} \mid \sigma\left\{B(s) ; s \le  t_{k}\right\}\right)\Big] \\
  = &\, \mathbb{E}\Big[\mathbb{I}_{G_{k}} \cdot \mathbb{P}\left(\left|I_{1}+I_{2}\right| \le  2 \eta \mid \sigma\left\{B(s) ; s \le  t_{k}\right\}\right)\Big] \\
   \le &\,  \mathbb{P}\left(G_{k}\right) \cdot \mathbb{P}\left(\left|I_{1}\right| \le  2 \eta\right) \\
  \le &\,  \mathbb{P}\left(G_{k}\right) \rho.
\end{split}
\end{equation*}
 By induction,  we obtain  \eqref{eq Gk}  with $\rho=\Phi\left(\frac{2 \sqrt{2 \beta} \Gamma(\alpha+\theta+1)}{\Gamma(\alpha+1)}\right)-\Phi\left(-\frac{2 \sqrt{2 \beta} \Gamma(\alpha+\theta+1)}{\Gamma(\alpha+1)}\right)$.
  The proof is complete.
\end{proof}

\subsection{Another maximal inequality}

\begin{proposition}\label{Prop-MtMu1}
Assume $\gamma\in [0,1),   \alpha\in \left(-\frac12+\frac{\gamma}{2}, \,  \frac12+\frac{\gamma}{2} \right)$ and $\theta>0$.  Let   $\tau\in (0, \frac{1-H}{3}] \cap (0, \frac{1-\gamma}{12}]\cap (0, \frac{H}{6}]$. Then, for any  $u>t>0$ and $\nu,\eta \in (0, 2/K_2)$, we have
\begin{equation}\label{4.10}
\begin{split}
&\mathbb{P}\left(M(t) \leq \nu t^{H+\theta}, M(u) \leq \eta u^{H+\theta}\right)\\
 \leq &\,\,  \varphi (\nu) \varphi (\eta) \exp \left(c_{3,2}\left(\frac{t}{u}\right)^{\tau}\left(\frac{1}{\nu^{1+1/\beta}}+\frac{1}{\eta^{1+1 /\beta}}\right)\right)
 +{c_{3,3}} \exp \left(-c_{3,4}\left(\frac{u}{t}\right)^{\tau}\right).
\end{split}
\end{equation}
Here, the constants $c_{3,2}, c_{3,3}, c_{3,4}$ do not  depend   on $u, t, \nu$ and $\eta$.
\end{proposition}

 Recall the following two lemmas about the moment estimates and the maximal inequality for the fractional integrated Gaussian processes.

\begin{lemma}{\rm\cite[Lemma 5.4]{El03}}\label{5.4}
Let  $\{X(t)\}_{0 \le  t \le  1}$  be a separable, centered, real-valued Gaussian process with  $X(0)=0$. Assume that there exists some  positive constants $\eta$ and $c_X$  such that
$$
\left(\mathbb{E}(X(t+h)-X(t))^{2}\right)^{1 / 2} \le  c_{X} h^{\eta}, \quad  h\in (0,1).
$$
Let
\begin{equation}\label{fra}
Y(s)=\frac{1}{\Gamma(\theta)} \int_{0}^{s}(s-w)^{\theta-1} X(w) \mathrm{d} w, \quad 0 \le  s \le  1,\, \theta>0.
\end{equation}
Then,  we have
$$
\left(\mathbb{E}(Y(s+h)-Y(s))^{2}\right)^{1 / 2} \le  c_{3,5}c_{X}\max \left(h^{\theta / 2}, h^{1 / 2}\right),
$$
where the constant $c_{3,5}$     depends  only on  $\eta$  and  $\theta$.
\end{lemma}
\begin{lemma}{\rm\cite[Lemma 5.5]{El03}}\label{5.5}
Let  $\{Y(s)\}_{0 \le  s \le  1}$  be the fractional integral defined in $\eqref{fra}$. If there exist some positive constants $\eta$ and $c_Y$ such that
\begin{align}
\left(\mathbb{E}(Y(s+h)-Y(s))^{2}\right)^{1 / 2}\le  c_{Y} h^{\eta}, \quad  h\in (0,1),
\end{align}
then for any $\delta>c_{Y}$,
$$
\mathbb{P} \left(\sup _{0 \le  s \le  1}|Y(s)| \ge  \delta\right) \le  \frac{1}{c_{3,6}} \exp \left(-c_{3,6}\left(c_{Y}^{-1} \delta\right)^{2}\right),
$$
where  $c_{3,6}$ is a positive constant independent of  $c_{Y}$  and  $\delta$.
\end{lemma}

\begin{proof}[Proof of Proposition \ref{Prop-MtMu1}]
It is sufficient to prove \eqref{4.10} when $u/t$ is large enough, since  \eqref{4.10} always true for some constants   when $u/t$ is bounded.
We set $v=\sqrt{ut}$. The idea is that if $t\ll u$, then $t\ll v\ll u$.
We recall the GFBM $X$ in \eqref{eq X} and define
$$G(s, x):=\left((s-x)_{+}^{\alpha}-(-x)_{+}^{\alpha}\right)|x|^{-\gamma/2} , \quad\text{for}\,  s,\, x \in \mathbb{R}. $$
Then
\begin{equation*}
\begin{split}
X(s) &=\int_{|x| \leq v} G(s, x) B(d x)+\int_{|x|> v} G(s, x) B(d x) \\
&=:X_{1}(s)+X_{2}(s).
\end{split}
\end{equation*}
The processes  $X_1$  and  $X_2$  are independent.
Let  $$Y_{1}(s):=\frac{1}{\Gamma(\theta)}\int_{0}^{s}(s-w)^{\theta-1} X_{1}(w) d w,$$
$$ Y_{2}(s):=\frac{1}{\Gamma(\theta)}\int_{0}^{s}(s-w)^{\theta-1} X_{2}(w) d w.$$
Then  $Y =Y_{1} +Y_{2}$, and the processes  $Y_1$  and $ Y_{2} $  are also independent.

For any  $\delta>0$, we have
\begin{equation}\label{5}
\begin{split}
& \mathbb P\left(M(t) \le  \nu t^{H+\theta}, M(u)\leq \eta u^{H+\theta}\right) \\
= &\, \mathbb P\left(\sup _{0 \leq s \leq t}|Y(s)| \le  \nu t^{H+\theta}, \sup _{0 \leq s \leq u}|Y(s)| \le  \eta u^{H+\theta}\right) \\
\leq &\, \mathbb P\left(\sup _{0 \leq s \leq t}\left|Y_{1}(s)\right| \leq(\nu+\delta) t^{H+\theta}, \sup _{0 \leq s \leq u}\left|Y_{2}(s)\right| \leq(\eta+\delta) u^{H+\theta}\right) \\
 &\, +\mathbb P\left(\sup _{0 \leq s \leq t}\left|Y_{2}(s)\right| \ge  \delta t^{H+\theta}\right)+\mathbb P\left(\sup _{0 \leq s \leq u}|Y_{1}(s)| \ge  \delta u^{H+\theta}\right)\\
=:&\, J_{1}+J_{2}+J_{3}.
\end{split}
\end{equation}
By the independence of  $Y_{1} $  and  $Y_{2} $, we have
\begin{equation}\label{3}
J_{1}=\mathbb P\left(\sup _{0 \leq s \leq t}\left|Y_{1}(s)\right| \leq(\nu+\delta) t^{H+\theta}\right) \cdot \mathbb P\left(\sup _{0 \leq s \leq u}\left|Y_{2}(s)\right| \leq(\eta+\delta) u^{H+\theta}\right).
\end{equation}
Notice that
\begin{equation}
\mathbb P\left(\sup _{0 \leq s \leq t} |Y_{1}(s)| \le (\nu+\delta) t^{H+\theta}\right) \le  \mathbb P\left(\sup _{0 \leq s \leq t} |Y(s)| \le (\nu+2 \delta)t^{H+\theta}\right)+J_{2},
\end{equation}
and
\begin{equation}\label{4}
\mathbb P\left(\sup _{0 \leq s \leq u}\left|Y_{2}(s)\right| \le (\eta+\delta) u^{H+\theta}\right) \leq \mathbb P\left(\sup _{0 \leq s \leq u}|Y(s)| \le (\eta+2 \delta) u^{H+\theta}\right)+J_{3}.
\end{equation}
Plugging \eqref{3} to \eqref{4} into \eqref{5}, we obtain
\begin{equation}\label{new}
\mathbb P\left(M(t) \le  \nu t^{H+\theta}, M(u)\leq \eta u^{H+\theta}\right) \leq \varphi (\nu+2 \delta) \varphi (\eta+2 \delta)+3 J_{2}+3 J_{3}.
\end{equation}
For any $\nu, \eta \in (0, 2/K_2)$,
by Lemma \ref{cor4.1}, we have
\begin{equation}\label{(1)}
\varphi (\nu+2 \delta)\varphi (\eta+2 \delta)\leq \varphi (\nu)\varphi (\eta)\exp \left(2 \delta K_{3}\left(\frac{1}{\nu^{1+1/\beta}}+\frac{1}{\eta^{1+1/\beta}}\right) \right).
\end{equation}
Next, it remains  to obtain the  upper bounds  of $J_2$ and $J_3$.

By a change of variable, we have
 \begin{equation*}
\begin{split}
J_{2}&=\, \mathbb{P}\left(\sup _{0 \leq s \leq t}\left|Y_{2}(s)\right|\geq\delta t^{H+\theta}\right)\\
&=\mathbb{P}\left(\sup _{0 \leq s \leq t}\left| \frac{1}{\Gamma(\theta)}\int_{0}^{s}(s-y)^{\theta-1}d y  \int_{|x|>v}B(d x)\left((y-x)_{+}^{\alpha}-(-x)_{+}^{\alpha}\right)| x|^{-\frac{\gamma}{2}}  \right| \geq \delta t^{H+\theta}\right)\\
&=\mathbb{P}\left(\sup _{0 \leq s \leq 1}\left|\frac{1}{\Gamma(\theta)}\int_{0}^{s}(s-y)^{\theta-1} d y \int_{|x|>\frac{v}{t}}B(d x)\left((y-x)_{+}^{\alpha}-(-x)_{+}^{\alpha}\right)| x|^{-\frac{\gamma}{2}}  \right| \geq \delta\right).
\end{split}
\end{equation*}
Denote
  \begin{equation*}
\begin{split}
 \overline{X}_{2}(y):=&\, \int_{|x|>\frac{v}{t}} \left((y-x)_{+}^{\alpha}-(-x)_{+}^{\alpha}\right)|x|^{-\frac{\gamma}{2}} B(d x),\\
\overline{Y}_{2}(s):= &\, \frac{1}{\Gamma(\theta)}\int_{0}^{s}(s-y)^{\theta-1} \overline{X}_{2}(y) d y.
\end{split}
\end{equation*}
By using the following  elementary inequality,
\begin{equation}\label{e1}
 \left|(x+y)^{\alpha}-x^{\alpha}\right| \le |\alpha| x^{\alpha-1} y,  \quad \text{for any $ x, y>0$ and $\alpha \le 1$,}
\end{equation}
 we have that for any  $0\leq s<s+w\leq1$,
\begin{equation*}
\begin{split}
\mathbb{E}\left[\left(\overline{X}_{2}(s+w)-\overline{X}_{2}(s)\right)^{2}\right]&=
\int_{-\infty}^{-v/t}\left((s+w-x)^{\alpha}-(s-x)^{\alpha}\right)^{2}(-x)^{-\gamma}dx\\
&\leq \int_{-\infty}^{-v/t}\alpha^{2} w^{2}(-x)^{2\alpha-2-\gamma}dx\\
&=\frac{\alpha^{2}}{2-2H} \left(\frac{v}{t}\right)^{2 H-2} w^{2}.
\end{split}
\end{equation*}
By Lemma \ref{5.4}, there exists a positive constant $C_1>0$ such that
\begin{equation}\label{c1}
\mathbb{E}\left[\left(\overline{Y}_{2}(s+w)-\overline{Y}_{2}(s)\right)^{2}\right] \leq C^{2}_{1}\left(\frac{v}{t}\right)^{2 H-2}\max\left\{w^{\theta},w\right\}.
\end{equation}
Applying  Lemma  \ref{5.5} with  $\eta=\min\left\{\frac{\theta}{2}, \frac12\right\}>0$, $c_{Y}=C_1\left(\frac{v}{t}\right)^{H-1}$, we have that for  any $\delta>c_{Y}$,
\begin{equation}\label{J2}
J_{2}=\mathbb{P}\left(\sup _{0 \leq s \leq 1}\left|\overline{Y}_{2}(s)\right|\geq\delta\right) \leq \frac{1}{c_{3,6}} \exp \left(-\frac{c_{3,6}}{C^{2}_{1}\left(v^{2}/t^{2}\right)^{H-1}} \delta^{2}\right).
 \end{equation}
 Let $\delta=\left(t/u\right)^{\tau}$. Taking  $u/t \geqslant C_1^{6/(1-H)}$  and  $\tau \leq (1- H)/3$ such that $\delta>c_{Y}$,  we have that by \eqref{J2},
\begin{equation}\label{(2)}
J_{2}\leq \frac{1}{c_{3,6}} \exp \left(-\frac{c_{3,6}}{C^{2}_{1}}\left(\frac{u}{t}\right)^{1-H-2\tau} \right)\leq \frac{1}{c_{3,6}}\exp \left(-\frac{ c_{3,6}}{C^{2}_{1}}\left(\frac{u}{t}\right)^{\tau} \right).
 \end{equation}

For the third term,  by  a change of variable, we have
\begin{align*}
 J_{3}=&\,\mathbb{P}\left(\sup _{0 \leq s \leq u}\left|Y_{1}(s)\right|\geq\delta u^{H+\theta}\right)= \mathbb{P}\left(\sup _{0 \leq s \leq 1}\left|\overline{Y}_{1}(s)\right|\geq\delta\right),
 \end{align*}
 where $$
\overline{Y}_{1}(s)=\frac{1}{\Gamma(\theta)}\int_{0}^{s}(s-y)^{\theta-1} \overline{X}_{1}(y) d y,$$
with
 $$\overline{X}_{1}(y)=\int_{|x|\le \frac{v}{u}} \left((y-x)_{+}^{\alpha}-(-x)_{+}^{\alpha}\right)|x|^{-\frac{\gamma}{2}} B(d x).$$
For the upper bound  of $J_3$, we have that for any $0\le  s<s+w\le  1$,
\begin{align*}
& \overline{X}_{1}(s+w)-\overline{X}_{1}(s)\\
=&\, \int_{-v/u}^{(s+w)\wedge (v/u)}(s+w-x)^{\alpha}|x|^{-\frac{\gamma}{2}} B(dx)-\int_{-v/u}^{s\wedge (v/u)}(s-x)^{\alpha}|x|^{-\frac{\gamma}{2}} B(dx).
\end{align*}
Consequently,
\begin{equation}\label{eq X1 diff}
\begin{split}
\mathbb{E}\left[\left(\overline{X}_{1}(s+w)-\overline{X}_{1}(s)\right)^{2}\right]
=&\int_{s\wedge (v/u)}^{(s+w)\wedge (v/u)}(s+w-x)^{2\alpha}|x|^{-\gamma}dx\\
&+\int_{-v/u}^{s\wedge (v/u)}((s+w-x)^{\alpha}-(s-x)^{\alpha})^{2}|x|^{-\gamma}dx\\
=:&I_{1}+I_{2}.
\end{split}
\end{equation}
 When  $s\ge \frac{v}{u} $, $I_{1}=0.$  When  $s<\frac{v}{u}$, taking
\begin{equation}
    \label{def r}
    r := \max\left\{\frac{1}{2}, \frac{\alpha}{2H}+ \frac{1}{2}\right\} \in \left(\frac{\alpha}{H},1\right),
\end{equation}
we have that $rH>0$ and $\alpha-rH<0$. Consequently,
 \begin{equation}\label{eq X1 I1}
\begin{split}
I_{1}&=\int_{s}^{(s+w) \wedge (v / u)}(s+w-x)^{2rH+2\alpha-2rH} x^{-\gamma} d x \\
&\leqslant w^{2 r H} \int_{s}^{(s+w) \wedge (v / u)}(s+w-x)^{2 \alpha-2 r H} x^{-\gamma} d x \\
&\leq w^{2 r H} \int_{0}^{(s+w) \wedge(v / u)}((s+w) \wedge( v / u)-x)^{2 \alpha-2 r H} x^{-\gamma} d x \\
&\leqslant \mathcal{B}(1-\gamma, 2 \alpha-2 r H+1) w^{2 r H} \left(\frac{v}{u}\right)^{2 H-2 r H}, \\
\end{split}
\end{equation}
where $\mathcal{B}(p,q)=\int_{0}^{1}x^{p-1}(1-x)^{q-1}dx,  (p>0, q>0)$, denotes the Beta function.

  By using the following  elementary inequality
   \begin{equation}\label{e2}
 \left|(x+y)^{\alpha}-x^{\alpha}\right| \leq y^{\alpha}, \quad \text{for any } x, y\ge0, \, \text{and} \, \alpha \le 1,
 \end{equation}
 we have that for any $s, w\ge0$ and $x\le s$,
\begin{equation*}
\begin{split}
\left|(s+w-x)^{\alpha}-(s-x)^{\alpha}\right|&=\left|\alpha \int_{s}^{s+w}(y-x)^{\alpha-1} d y\right|\\
&=\left|\alpha \int_{s}^{s+w}(y-x)^{r H-1} (y-x)^{\alpha-r H} d y\right|\\
&\leqslant\left|\frac{\alpha}{ rH}(s-x)^{\alpha-rH}\left((s+w-x)^{r H}-(s-x)^{r H}\right)\right|\\
& \leqslant \frac{| \alpha|}{r H}(s-x)^{\alpha-r H} w^{r H}.
\end{split}
\end{equation*}
 It follows that
\begin{equation*}
\begin{split}
I_{2}=&\, \int_{-v/u}^{0}((s+w-x)^{\alpha}-(s-x)^{\alpha})^{2}|x|^{-\gamma}dx\\
&+\int_{0}^{s\wedge (v/u)}((s+w-x)^{\alpha}-(s-x)^{\alpha})^{2}|x|^{-\gamma}dx\\
 \leq &\,  \left(\frac{\alpha}{ r H}\right)^{2} w^{2 r H}\left[\int_{0}^{s \wedge (v / u)}(s-x)^{2 \alpha-2 rH} x^{-\gamma} d x+\int_{0}^{s \wedge (v / u)}(s+x)^{2 \alpha-2 rH} x^{-\gamma} d x\right].
\end{split}
\end{equation*}
Since $\alpha-rH<0$ and $H-rH>0$, we get
 \begin{equation}\label{eq X1 I2}
\begin{split}
I_{2}
&\leq\left(\frac{\alpha}{ r H}\right)^{2} w^{2 rH}\left[\int_{0}^{s \wedge (v / u)}(s \wedge (v / u)-x)^{2 \alpha-2 r H} x^{-\gamma} d x+\int_{0}^{v / u} x^{2 \alpha-2 r H-\gamma} d x\right] \\
&\leqslant \left(\frac{\alpha}{ r H}\right)^{2}w^{2 r H}\left[\mathcal{B}(1-\gamma, 2 \alpha-2r H +1)\left(\frac{v}{u}\right)^{2 H-2 r H}+\frac{1}{2 H-2 r H}\left(\frac{v}{u}\right)^{2 H-2 r H}\right]. \\
\end{split}
\end{equation}
Putting \eqref{eq X1 diff}, \eqref{eq X1 I1} and \eqref{eq X1 I2} together, there exists a positive constant $c_{3,7}$ such that
$$\mathbb{E}\left(\overline{X}_{1}(s+w)-\overline{X}_{1}(s)\right)^{2}\leqslant c_{3,7} \left(\frac{v}{u}\right)^{2 H-2 r H} w^{2 r H}.$$
By Lemma \ref{5.4}, there exists a positive constant $C_2>0$ such that
\begin{equation}\label{c2}
\mathbb E\left(\overline{Y}_{1}(s+w)-\overline{Y}_{1}(s)\right)^{2}  \le  C_2^{2}\left(\frac{v}{u}\right)^{2 H- 2rH}\max\left\{w^{\theta},w\right\}.
\end{equation}
 Consequently, Applying Lemma \ref{5.5} with $\eta=\min\{\theta/2 ,1/2\}>0$, $c_{Y}=C_2 \left(v/u\right)^{H-rH}$, we have that for any $\delta>c_{Y}$,
\begin{equation}\label{J3}
J_{3}=\mathbb P\left(\sup _{0 \leq s \leq 1}\left|\overline{Y}_{1}(s)\right|\geq\delta\right) \le  \frac{1}{c_{3,6}} \exp \left(-\frac{c_{3,6}}{C_2^{2}\left(v^{2}/u^{2}\right)^{ H-rH}} \delta^{2}\right).
\end{equation}
 Taking $u/t \geqslant C_2^{6/(H-rH)}$ and $\tau \le  \frac{H-rH}{3}$ such that $\delta = \left(t/u\right)^{\tau} >c_{Y}$. Recalling the definition of $r$ by \eqref{def r}, it is sufficient to take
 $u/t \geqslant \max \{C_2^{12/H}, C_2^{24/(1-\gamma)} \}$ and $\tau \le  \min\{ \frac{H}{6}, \frac{1-\gamma}{12} \}$. Consequently,  we have that by \eqref{J3},
\begin{equation}\label{(3)}
J_{3}\le  \frac{1}{c_{3,6}} \exp \left(-\frac{c_{3,6}}{C_2^{2}}\left(\frac{u}{t}\right)^{H-rH-2\tau}\right) \le  \frac{1}{c_{3,6}} \exp \left(-\frac{c_{3,6}}{C_2^{2}}\left(\frac{u}{t}\right)^{\tau}\right).
\end{equation}
Putting  \eqref{new}, \eqref{(1)}, \eqref{(2)} and \eqref{(3)} together,  we get \eqref{4.10}.

In conclusion, we have proved that \eqref{4.10} holds when $u/t \geqslant \max\left\{C_1^{6/(1-H)},C_2^{12/H}, C_2^{24/(1-\gamma)}\right\}$ with constants $C_1$ and $C_2$ defined in \eqref{c1} and \eqref{c2} respectively.
Since \eqref{4.10} is always true for some  constants $c_{3,2}, c_{3,3}, c_{3,4}$    uniformly over $1\leqslant u/t \leqslant \max\left\{C_1^{6/(1-H)},C_2^{12/H}, C_2^{24/(1-\gamma)}\right\}$, the proof is complete.
\end{proof}

\section{Proof of Theorem \ref{Thm-ChungLIL}}
\label{Proof of Thm-ChungLIL}

To prove Theorem \ref{Thm-ChungLIL}, we apply  some zero-one laws to show that  the lower limits in \eqref{3.1} and \eqref{3.2} are  constants almost surely (possibly $0$ or $\infty$).  Then we use  Theorems  \ref{Thm-LowerClassesZero} and \ref{Thm-LowerClassesInf} to prove  that the constants are positive and finite.
  Following Takashima \cite{Tak89}, we say that
an $(H+\theta)$-self-similar process $Y=\{Y(t)\}_{t\ge0}$ is ergodic (resp. strong mixing) if for every $a > 0$ and $a \neq 1$, the
scaling transformation $S_{a}(Y)=\{a^{-(H+\theta)}Y(at)\}_{t\ge0}$ is ergodic (resp. strong mixing). This is equivalent to saying that all
the shift transformations for the corresponding stationary process obtained via Lamperti's transformation
$L(Y) =\{e^{-(H+\theta)t}Y(e^{t}
)\}_{t\in \mathbb{R}}$ are ergodic (resp. strong mixing).

We recall a zero-one law for ergodic self-similar processes, which complements the results of Takashima \cite{Tak89}.
\begin{proposition}{\rm\cite[Proposition 3.3]{TX07}}
\label{TX07-Prop3.3}
    Let $X = \{ X(t)\}_{t\in \mathbb{R}}$ be a separable, self-similar process with index $\kappa$. We assume that $X(0) = 0$ and that $X$ is ergodic. Then, for any increasing function $\psi: \mathbb{R}_+  \rightarrow \mathbb{R}_+$, we have $\mathbb{P}(E_{\kappa, \psi}) = 0$ or $1$, where
    \begin{equation*}
        E_{\kappa, \psi}:= \left\{ \omega: \, \text{there exists $\delta > 0$ such that $\sup_{0 \le s \le t} |X(s)| \ge t^{\kappa} \psi(t)$ for all $0 < t \le \delta$}\right\}.
    \end{equation*}
\end{proposition}

We then use Proposition \ref{TX07-Prop3.3} to prove a weaker form of Theorem \ref{Thm-ChungLIL}.

\begin{lemma}\label{le3}
Assume  $\gamma \in[0,1)$, $\alpha \in(-1 / 2+\gamma/2, 1 / 2)$ and $\theta>0$.
\begin{itemize}
\item[(a)]  There exists a  constant  $\kappa_{1}^{\prime}\in[0, \infty]$  such that
\begin{equation}\label{3.6}
\liminf _{t \rightarrow 0+} \sup _{0 \le  s \le  t} \frac{|Y(s)|}{t^{H+\theta} /\left(\ln \ln t^{-1}\right)^{\beta }}=\kappa_{1}^{\prime} \quad \text { a.s. }
\end{equation}\label{3.7a}

\item[(b)]  There exists a  constant  $\kappa_{2}^{\prime}\in[0, \infty]$  such that
\begin{equation}\label{3.7b}
\liminf _{t \rightarrow \infty} \sup _{0 \le  s \le  t} \frac{|Y(s)|}{t^{H+\theta} /(\ln \ln t)^{\beta}}=\kappa_{2}^{\prime} \quad \text { a.s. }
\end{equation}
\end{itemize}
\end{lemma}
\begin{proof}
The proof  is inspired by  Tudor and Xiao \cite{TX07}. First, we show that the auto-covariance function of $L(Y)$ satisfies
\begin{align}\label{3.13}
e^{-(H+\theta) t} \mathbb{E}\left[Y(e^{t}) Y(1)\right]=O\left(e^{-\kappa t}\right) \text { as } t \rightarrow +\infty,
\end{align}
where $\kappa>0$ is a constant.

Since
\begin{equation}
\begin{split}\label{ly}
&e^{-(H+\theta) t} \mathbb{E}\left[Y(e^{t}) Y(1)\right]\\
=&\, e^{-(H+\theta) t} \mathbb{E}\left(\frac{1}{\Gamma^{2}(\theta)} \int_{0}^{e^{t}} \left(e^{t}-v\right)^{\theta-1} X(v)d v \cdot \int_{0}^{1}(1-u)^{\theta-1} X(u)d u \right) \\
=&\, \frac{1}{\Gamma^{2}(\theta)} e^{-(H+\theta) t} \int_{0}^{e^{t}} \left(e^{t}-v\right)^{\theta-1}d v\int_{0}^{1}d u(1-u)^{\theta-1} \mathbb{E}\left[X(u) X(v)\right]. \\
\end{split}
\end{equation}
For  $0\le  u\le  1$ and $0\le  v\le  e^{t}$,
\begin{equation*}
\begin{split}
\mathbb{E}\left[X(u) X(v)\right]=&\, \int_{0}^{u \wedge v}(u-x)^{\alpha}(v-x)^{\alpha} x^{-\gamma} d x\\
&+\int_{0}^{+\infty}\left((u+x)^{\alpha}-x^{\alpha}\right)\left((v+x)^{\alpha}-x^{\alpha}\right) x^{-\gamma} d x \\
=:&\, J_{1}+J_{2}.
\end{split}
\end{equation*}

For the integral  $J_{1}$, when  $\alpha \in\left[0, 1/2\right)$, for  $0\le  u\le  1$ and $0\le  v\le  e^{t}$,
\begin{equation*}
J_{1} \le  u^{\alpha} v^{\alpha} \int_{0}^{u \wedge v} x^{-\gamma} d x=\frac{1}{1-\gamma}u^{\alpha} v^{\alpha}(u \wedge v)^{1-\gamma}\le  \frac{1}{1-\gamma}u^{\alpha} v^{\alpha}.
\end{equation*}
When  $\alpha \in\left(-1/2+\gamma/2, 0\right)$,
\begin{equation*}
\begin{split}
J_{1}\le &\,  \int_{0}^{u \wedge v}((u \wedge v)-x)^{2 \alpha} x^{-\gamma} d x\\
=&\,  (u \wedge v)^{2 H}\mathcal{B}(1-\gamma, 2 \alpha+1)\le  \mathcal{B}(1-\gamma, 2 \alpha+1).
\end{split}\end{equation*}
Therefore
\begin{equation*}
J_1\leqslant \max\left\{ \frac{1}{1-\gamma}u^{\alpha} v^{\alpha}, \mathcal{B}(1-\gamma, 2 \alpha+1)\right\}.
\end{equation*}

To deal with the integral  $J_{2}$, we divide the integral interval as follows:
\begin{equation*}
    \begin{split}
J_2&= \left[\int_{0}^{1}+\int_{1}^{+\infty}\right]\left((u+x)^{\alpha}-x^{\alpha}\right)\left((v+x)^{\alpha}-x^{\alpha}\right) x^{-\gamma} d x \\
&=:J_{2,1}+J_{2,2}.
    \end{split}
\end{equation*}
By \eqref{e2},
we have
\begin{equation*}
J_{2,1} \leq \int_{0}^{1} u^{\alpha} v^{\alpha} x^{-\gamma} d x = \frac{1}{1-\gamma}u^{\alpha} v^{\alpha}.
\end{equation*}
For any
$\delta \in\left((\alpha-\gamma)_{+}/(1-\alpha),\left(1/2-\gamma/2\right) /(1-\alpha)\right)$, by \eqref{e1},  we have
\begin{equation*}
J_{2,2} \leq |\alpha| u \int_{1}^{+\infty}\left((v+x)^{\alpha}-x^{\alpha}\right)x^{\alpha-1-\gamma} d x.
\end{equation*}
Then by using inequalities \eqref{e1} and  \eqref{e2} for $\left((v+x)^{\alpha}-x^{\alpha}\right)^{\delta}$ and $\left((v+x)^{\alpha}-x^{\alpha}\right)^{(1-\delta)}$ respectively, we obtain
\begin{equation*}
\begin{split}
J_{2,2} \leqslant
&\,  |\alpha|u \int_{1}^{+\infty}\left(|\alpha| v x^{\alpha-1}\right)^{\delta}\left(v^{\alpha}\right)^{1-\delta} x^{\alpha-1-\gamma} d x \\
 = &\, \frac{|\alpha|^{1+\delta}}{(1-\alpha) \delta+\gamma-\alpha} u v^{(1-\alpha) \delta+\alpha}.
\end{split}
\end{equation*}
Therefore
\begin{equation*}
J_2 \le  \frac{1}{1-\gamma}u^{\alpha} v^{\alpha}+\frac{|\alpha|^{1+\delta}}{(1-\alpha) \delta+\gamma-\alpha} u v^{(1-\alpha) \delta+\alpha}.
\end{equation*}

It follows that
\begin{equation}\label{e}
\begin{split}
\mathbb{E}\left[X(u)X(v)\right] &\le \frac{2}{1-\gamma}  u^{\alpha}v^{\alpha}+\frac{| \alpha|^{1+\delta}}{(1-\alpha) \delta+\gamma-\alpha}  u v^{(1-\alpha) \delta+\alpha}+\mathcal{B}(1-\gamma, 2 \alpha+1)\\
&\leqslant c_{4,1}\left(u^{\alpha} v^{\alpha}+u v^{(1-\alpha) \delta+\alpha}+1\right),
\end{split}
\end{equation}
where $c_{4,1}$ is a positive constant.

Plugging \eqref{e} into \eqref{ly}, we obtain
\begin{equation*}
\begin{split}
e^{-(H+\theta) t} \mathbb{E}\left[Y(e^{t}) Y(1)\right]\le & \,
\frac{c_{4,1}e^{-(H+\theta) t}}{\Gamma^{2}(\theta)}\int_{0}^{e^{t}} \left(e^{t}-v\right)^{\theta-1}v^{\alpha}d v\int_{0}^{1} (1-u)^{\theta-1}u^{\alpha}d u\\
\,&+\frac{c_{4,1}e^{-(H+\theta) t}}{\Gamma^{2}(\theta)}\int_{0}^{e^{t}}\left(e^{t}-v\right)^{\theta-1}v^{(1-\alpha) \delta+\alpha} d v\int_{0}^{1}(1-u)^{\theta-1}u d u\\
\,&+\frac{c_{4,1}e^{-(H+\theta) t}}{\Gamma^{2}(\theta)}
\int_{0}^{e^{t}}\left(e^{t}-v\right)^{\theta-1}d v \int_{0}^{1}(1-u)^{\theta-1}d u \\
\le &\, c_{4,2}
\left( e^{-(H-\alpha) t}+ e^{-\left[H-\alpha-(1-\alpha) \delta\right] t}+e^{-Ht} \right) \\
= &\, O\left(e^{-\kappa t}\right) \text { as } t \rightarrow +\infty,
\end{split}
\end{equation*}
where  $\kappa=\min \left\{\alpha+1/2-\gamma/2, 1/2-\gamma/2-\delta(1-\alpha)\right\}>0$.
By the Fourier inversion formula, $L(Y)$ has a continuous spectral density function. It
follows from the Theorem 8 in \cite{Maruyama1949THEHA} that $L(Y)$ is strong mixing and thus is ergodic.

Hence, applying Proposition \ref{TX07-Prop3.3} with  $\psi_{c}(t):=c(\ln \ln 1 / t)^{-\beta}$  and    $$\kappa^{\prime}_{1}:=\sup \left\{c \geq 0: \mathbb{P}\left(E_{\kappa, \psi_{c}}\right)=1\right\}, $$ we obtain  \eqref{3.6}.

The proof of \eqref{3.7b} is similar to that of \cite[Proposition 3.3]{TX07} with minor modifications.
 The proof is complete.
\end{proof}

\begin{proof}[Proof of Theorem \ref{Thm-ChungLIL}]
For any   $\lambda>0$, let
\begin{equation*}
\xi_{\lambda}(t):=\frac{\lambda t^{H+\theta}}{(\ln |\ln t|)^{\beta}}, \quad t>0.
\end{equation*}
By Theorem \ref{Thm-SmallBall} and Theorem \ref{Thm-LowerClassesZero} (resp. Theorem \ref{Thm-LowerClassesInf}), we derive that if  $\lambda<\kappa_{3}^{\beta}$,  $\xi_{\lambda} \in L L C_{0}(M)$ (resp. $\xi_{\lambda} \in L L C_{\infty}(M)$), else if  $\lambda> \kappa_{4}^{\beta}$, $\xi_{\lambda} \in   L U C_{0}(M)$(resp. $\xi_{\lambda} \in L U C_{\infty}(M)$). Then applying Lemma \ref{le3}, we obtain the desirable results.
\end{proof}

\medskip

\bigskip


\vskip0.5cm

\begin{thebibliography}{abc}

%
\bibitem{And1955}
Anderson T W. The integral of a symmetric unimodal function over a symmetric convex
set and some probability inequalities. {Proc Amer Math Soc}, 1955,  {\bf 6}: 170--176

\bibitem{Bin86}
 Bingham N H. Variants on the law of the iterated logarithm. Bull London Math Soc, 1986, {\bf 18}(5): 433--467

\bibitem{2}
Borell C.
\newblock {Convex measures on locally convex spaces}.
\newblock {Ark Mat}, 1974, {\bf 12}: 239--252

\bibitem{BM11}
 Buchmann  B,   Maller R. The small-time Chung-Wichura law for L\'evy processes with non-vanishing Brownian component.
 Probab Theory Related Fields, 2011, {\bf 149}(1-2): 303--330



\bibitem{CL03}
Chen X, Li W V.
\newblock Quadratic functionals and small ball probabilities for the $m$-fold
  integrated Brownian motion.
{Ann  Probab}, 2003, {\bf 31}(2): 1052--1077
%
\bibitem{Chung}
Chung K L.
\newblock On the maximum partial sums of sequences of independent random
  variables.
Trans Amer Math Soc, 1948, {\bf 64}: 205--233

\bibitem{El03}
El-Nouty C.
\newblock A note on the fractional integrated fractional Brownian motion.
\newblock {Acta Appl Math},  2003, {\bf 78}(1-3): 103--114
%
\bibitem{El04}
El-Nouty C.
\newblock Lower classes of integrated fractional brownian motion.
\newblock {Studia Sci Math Hungar}, 2004, {\bf 41}(1): 17--38

\bibitem{El11}
El-Nouty C. Lower classes of the Riemann-Liouville process. Bull Sci Math, 2011, {\bf 135}(1): 113--123

 \bibitem{El12}
 El-Nouty C.   The lower classes of the sub-fractional Brownian motion.
 In: Zili, M., Filatova, D. (eds) Stochastic Differential Equations and Processes. Springer Proceedings in Mathematics, vol 7. Springer, Berlin, Heidelberg, 2012


   \bibitem{IPT2020b}
   Ichiba T, Pang G D,  Taqqu M S. Semimartingale properties of a generalized  fractional Brownian motion
    and its mixtures with applications in finance. arXiv: 2012.00975


\bibitem{IPT22}
Ichiba T, Pang G D,  Taqqu M S. Path properties of a generalized fractional Brownian motion. J Theoret Probab, 2022, {\bf 35}(1): 550--574


\bibitem{Kue76}
 Kuelbs J. A strong convergence theorem for Banach space valued random variables. Ann Probab, 1976, {\bf 4}(5): 744--771


\bibitem{KuL93}
Kuelbs J, Li W V.  Metric entropy and the small ball problem for Gaussian
measures. J Funct Anal, 1993, {\bf 116}(1): 133--157




\bibitem{LT91}
Ledoux M, Talagrand M. Probability in a Banach Space. Berlin: Springer Press, 1991


\bibitem{Lif99}
Lifshits M A. Asymptotic behavior of small ball probabilities. Berlin, Boston: Prob Theory and
Math Stat Press, 1999

\bibitem{LL99}
Li W V,  Linde W.   Approximation, metric entropy and small ball estimates for Gaussian measures. {Ann  Probab}, 1999,  {\bf 27}(3): 1556--1578

\bibitem{LiS01}
Li W V,  Shao Q M. Gaussian processes: inequalities, small ball probabilities and
applications. Handbook of Statist, 2001, {\bf 19}: 533--597






\bibitem{Maruyama1949THEHA}
Maruyama G.
\newblock The harmonic analysis of stationary stochastic processes.
\newblock {Mem Fac Sci}, 1949, {\bf 4}: 45--106

\bibitem{MR95}
Monrad D,  Rootz\'en H. Small values of Gaussian processes and functional laws of the iterated
logarithm. { Probab Theory Related Fields}, 1995, {\bf 101}(2): 173--192









\bibitem{PT2019}
Pang G D, Taqqu M S. Nonstationary self-similar Gaussian processes as scaling limits of power-law
shot noise processes and generalizations of fractional Brownian motion. {High Freq}, 2019, {\bf 2}(2): 95--112

\bibitem{Rev90}
R\'ev\'esz P.
\newblock Random walk in random and non-random environments. Teaneck, NJ:
 World Scientific Press, 1990





\bibitem{Tak89}
Takashima K.
\newblock Sample path properties of ergodic self-similar processes.
\newblock {Osaka J Math}, 1989, {\bf 26}(1): 159--189


\bibitem{Tal96}
Talagrand M.
\newblock Lower classes for fractional Brownian motion.
\newblock {J Theoret Probab},  1996, {\bf 9}(1): 191--213

\bibitem{TX07}
Tudor C A, Xiao Y M.
\newblock {Sample path properties of bifractional Brownian motion}.
\newblock {Bernoulli}, 2007, {\bf 13}(4): 1023--1052







\bibitem{WX2022a}
Wang R, Xiao Y M. Exact uniform modulus of continuity and Chung's LIL for the generalized fractional Brownian motion.  {J  Theoret  Probab}, 2022, {\bf 35}(4):  2442--2479

\bibitem{WX2022b} Wang R, Xiao Y M.
Lower functions and Chung's LILs of the generalized fractional Brownian motion.   J Math Anal Appl, 2022, {\bf 514}(2): 126320


\bibitem{Xiao2008}
Xiao Y M.  Strong local nondeterminism and sample path properties of Gaussian random fields. Adv Lect Math, 2008, {\bf 2}: 136--176

\bibitem{ZL2006}
Zhang R, Lin Z.  A Functional LIL for $m$-Fold Integrated Brownian Motion. Chin  Ann  Math  Ser  B, 2006,  {\bf 27}, 459--472


\end{thebibliography}
\end{document}